\documentclass[10pt,twoside,reqno]{amsart}
\usepackage{hyperref}
\usepackage{amsmath}
\usepackage{amssymb}

\numberwithin{equation}{section}
	\newtheorem{theorem}{Theorem}[section]

	\newtheorem{proposition}[theorem]{Proposition}
	\newtheorem{lemma}[theorem]{Lemma}

	\allowdisplaybreaks

\begin{document}
	\title[Korteweg de-Vries-Kawahara equation]
	{On the Radius of Analyticity for a Korteweg-de Vries-Kawahara Equation with a Weakly Damping Term}
	
	\author[A. Boukarou]
	{Aissa Boukarou}
	\address{A. Boukarou \newline
	Department of Mathematics \newline
	Higher School of Technological Education, Skikda \newline
	Algeria}
	\email{boukarouaissa@gmail.com}
	
	\author[D. O. da Silva]{Daniel Oliveira da Silva}
	\address{D. O. da Silva \newline
	Department of Mathematics \newline
	Nazarbayev University \newline
	Qabanbai Batyr Avenue 53 \newline
	010000 Nur-Sultan \newline
	Kazakhstan \newline
	\newline
	Current affiliation: \newline
	Department of Mathematics \newline
	California State University, Los Angeles \newline
	5151 State University Drive \newline
	Los Angeles, CA 90032}
	\email{daniel.dasilva@nu.edu.kz}

	
	\begin{abstract}
	We consider the Cauchy problem for an equation of Korteweg-de Vries-Kawahara type with initial data in the analytic Gevrey spaces.  By using linear, bilinear and trilinear estimates in analytic Bourgain spaces, we establish the local well-posedness for this problem.  By using an approximate conservation law, we extend this to a global result in such a way that the radius of analyticity of solutions is uniformly bounded below by a fixed positive number for all time.
	\end{abstract}
	
	\maketitle


\section{Introduction}

In a recent paper \cite{Coclite}, Coclite and di Ruvo studied the Cauchy problem for the Korteweg-de Vries-Kawahara (KdV-K) equation
\begin{equation}\label{pp1}
\partial_{t}u+\alpha\partial_{x}^{5}u+\beta\partial_{x}^{3}u+\mu\partial_{x}u^{2}+\lambda\partial_{x}u^{3} = 0.
\end{equation}
Here, $x\in \mathbb{R}$, $t \geq 0 $, and the parameters $\alpha$, $\beta$, $ \lambda$ and $\mu$ are real numbers, with $\alpha \neq 0$.  In their work, they showed that this problem is globally well-posed for initial data in the Sobolev space $H^{4}$.  This result was obtained by considering a modification of this equation, known as the \emph{vanishing viscosity approximation}, which is given by
\[
\partial_{t}u_{\epsilon} + \alpha \partial_{x}^{5} u_{\epsilon} + \beta \partial_{x}^{3} u_{\epsilon} + \mu \partial_{x} u^{2}_{\epsilon} + \lambda \partial_{x} u_{\epsilon}^{3} = \epsilon \partial_{x}^{6} u_{\epsilon},
\]
and then passing to the limit $\epsilon \rightarrow 0$.

In the present work, we would like to consider the following modification to equation \eqref{pp1}:
\begin{equation}\label{p1}
\partial_{t} u + \alpha \partial_{x}^{5} u + \beta \partial_{x}^{3} u + \mu \partial_{x} u^{2} + \lambda \partial_{x} u^{3} + a(x) u = 0.
\end{equation}
Here, we have added the damping term $au$, whose coefficient $a(x)$ is a function independent of time.  For equation \eqref{p1}, we would like to consider the Cauchy problem for equation \eqref{p1} with initial data in the Gevrey spaces $G^{\sigma}(\mathbb{R})$, which are defined by the norm
\[
\| f \|_{G^{\sigma}(\mathbb{R})} =\| f \|_{G^{\sigma}} = \left( \int_{\mathbb{R}} e^{2 \sigma |\xi|} |\hat{f}(\xi)|^{2}\ d\xi \right)^{\frac{1}{2}}.
\]
As is well-known by now, functions in the Gevrey class satisfy the following Paley-Wiener theorem:
\begin{theorem}\label{wiener}
Let $\sigma > 0$.  Then, the following are equivalent:
\begin{enumerate}
\item $f \in G^{\sigma}(\mathbb{R})$;
\item $f$ is the restriction to the real line of a function $\tilde{f}$ which is holomorphic in the strip
\[
S_{\sigma} = \{ x + i y: x,y \in \mathbb{R}, |y| < \sigma \}
\]
and satisfies
\[
\sup_{|y| < \sigma}\| \tilde{f}(x + i y) \|_{L^2_{x}} < \infty.
\]
\end{enumerate}
\end{theorem}
\noindent A proof of this result can be found on page 174 of \cite{Katznelson}.  One problem which has been receiving increasingly greater attention is that of \emph{persistence of analyticity}: given initial data $u(0) \in G^{\sigma_{0}}(\mathbb{R})$ and $T > 0$, for what values of $\sigma(T)$ will $u$ belong to the solution space $$C\left([0,T];G^{\sigma(T)}(\mathbb{R})\right)?$$ We will consider this question for equation \eqref{p1}.

The analyticity problem has been studied for decades, starting with the work of Kato and Masuda \cite{Kato} in 1986, who studied a class of equations of the form $u_{t} = F(t,u)$ and showed that the radius of analyticity decays at most super exponentially.  Since then, many results have appeared in the literature along these lines, each presenting increasingly better rates of the decay for the radius of analyticity.  Currently, the best known result is that of Wang \cite{Wang}, who studied a KdV equation containing a damping term and showed that there exists a constant $\rho > 0$ such that $\sigma(T) \geq \rho$ for all $T > 0$.  Thus, the radius of analyticity of solutions is uniformly bounded away from zero for all time.

Our goal in this paper is two-fold.  First, we will show that for sufficiently short times, there will be no decay in analyticity.  This is the content of the following theorem, which is the first main result:
\begin{theorem} \label{cor2}
	Let $\sigma_{0} > 0$ and $u_{0} \in  G^{\sigma_{0}}(\mathbb{R})$.  Assume that the function $a(x)$ satisfies the following assumptions:
 \begin{itemize}
 	\item[(A1)] There exists $\gamma > 0$ so that
 	\[
 	a(x)\geq \gamma > 0
 	\]
 	for all $x \in \mathbb{R}$.
 	\item [(A2)] The function $a$ satisfies
 	\[
 	\sum_{k = 0}^{\infty} (k + 1)^{\frac{1}{4}} \frac{\sigma_{0}^{k}}{k!} \| \partial_{x}^{k} a \|_{L^{\infty}} < \infty.
 	\]
 \end{itemize}
	Then the Cauchy problem
	\begin{equation}\label{kdvk}
	    \begin{cases}
	        & \partial_{t} u + \alpha \partial_{x}^{5} u + \beta \partial_{x}^{3} u + \mu \partial_{x} u^{2} + \lambda \partial_{x} u^{3} + a(x)u = 0, \\
	        & u(x,0) = u_{0}(x)
	    \end{cases}
	\end{equation}
	is locally well-posed in $G^{\sigma_{0}}(\mathbb{R})$.  Thus, for each $u_{0} \in G^{\sigma_{0}}(\mathbb{R})$, there exists a time interval $I = [0, \delta]$ such that the solution $u(t)$ to the problem \eqref{kdvk} admits a holomorphic extension $\tilde{u}$ on $S_{\sigma_{0}}$ for each $t \in I$.
\end{theorem}

The second goal of this paper is to use the methods of Wang to prove that the radius of analyticity for any time $T > 0$ is bounded away from zero.  This is the content of the next theorem, which is the second main result of this paper.
\begin{theorem}\label{global}
	Assume that (A1) and (A2) hold, and let $u_{0}\in G^{\sigma_{0}}(\mathbb{R})$ for some $\sigma_{0} > 0$. Then there exists a number $\tilde{\sigma_{0}} > 0$ such that for any $T > 0$, the initial value problem \eqref{kdvk} has a solution $$u \in C\left([0,T]; G^{\sigma_{1}}(\mathbb{R})\right).$$  Moreover, we have the estimate
	\begin{equation}\label{exp}
	\|u(t)\|_{G^{\tilde{\sigma_{0}}}} \leq Ce^{- \frac{\gamma t}{2}}
	\end{equation}
	for all $t \geq 0$.  Here, $\tilde{\sigma_{0}}$ depends on $\| u_{0} \|_{G^{\sigma_{0}}}$ and $\sigma_{0}$, but not on $T$.
\end{theorem}

Theorems \ref{cor2} and \ref{global} will be proved in sections \ref{local} and \ref{sect1.5}, respectively.  To prove these, we first state preliminary concepts and results in section \ref{prel}.


\section{Preliminaries}\label{prel}

Before we begin our proofs, we must fix the notation to be used, as well as introduce some auxiliary spaces which will be used in the proofs.  We begin with notation.

\subsection{Notation} 

We use $a \lesssim_{c_{1}, \ldots, c_{n}}  b$ to denote
$a \leq C b$ for some constant $C > 0$ depending on $c_{1}, \ldots, c_{n}$.  If $C$ is an absolute constant, we shall write $A \lesssim B$. 
Next, we define notation for some pseudodifferential operators which will be used for our discussion.  If $f \in L^{2}(\mathbb{R}^{2})$, we denote the spatial Fourier transform by
\[
\hat{f}(\xi, t) = \int_{\mathbb{R}} f(x,t) e^{-i x \xi}\ dx,
\]
and the spacetime Fourier transform by
\[
\tilde{f}(\xi, \tau) = \int_{\mathbb{R}^{2}} f(x,t) e^{-i (x \xi + t \tau)}\ dx dt.
\]
We will also use the notation $\langle x \rangle = (1 + |x|)$.

\subsection{Function Spaces}

Now, we will present the elementary spaces used in this paper.  In addition to the Gevrey class $G^{\sigma}(\mathbb{R})$, we will also need to make use of some analytic versions of the Bourgain spaces $X_{s,b}(\mathbb{R}^{2})$.  More precisely, for $b \in \mathbb{R}$ and $\sigma > 0$, define $X_{\sigma,b}(\mathbb{R}^{2})$ to be the Banach space defined by the norm 
\begin{equation*} 
\displaystyle
\| u\|^{2}_{X_{\sigma,b}(\mathbb{R}^{2})} =\| u\|^{2}_{X_{\sigma,b}} = \int_{\mathbb{R}^{2}} e^{2\sigma|\xi |} (1 + |\tau + \phi(\xi)|)^{2b}|\tilde{u}(\xi, \tau)|^{2} d\xi d\tau,
\end{equation*}
where $\phi(\xi)=\alpha\xi^{5}-\beta\xi^{3}$.  In the special case $\sigma = 0$, we use the notation $X_{b}(\mathbb{R}^{2})$.  Also, for $T > 0$, the notation $X^{T}_{\sigma,b}(\mathbb{R}^{2})$ denotes the restricted analytic Bourgain space defined by
\begin{equation} \label{1.5}
\displaystyle
\| u\|_{X^{T}_{\sigma,b}(\mathbb{R}^{2})} =\| u\|_{X^{T}_{\sigma,b}} = \inf \left\lbrace \| v\|_{X_{\sigma,b}}: v = u \text{ on } [0,T)\times \mathbb{R} \right\rbrace. 
\end{equation}
As is well-known, $X_{\sigma,b}(\mathbb{R}^{2})$ embeds continuously into $ C\left([0,T], G^{\sigma} (\mathbb{R})\right)$, provided $b > 1/2$. 
\begin{lemma}\label{cha1lem1}
	Let $b>\frac{1}{2} $ and  $\sigma > 0$.  Then for all $u \in X_{\sigma,b}(\mathbb{R}^{2})$, we have
	\begin{equation*}
	\|u \|_{L^{\infty}_{t}G^{\sigma}_{x}} \lesssim_{b} \|u\|_{X_{\sigma,b}}.
	\end{equation*}
\end{lemma}

In addition to the analytic Bourgain spaces, we will also need to introduce a special class of functions, denoted by $A^{\sigma}$, which is defined by the norm
\[
\| f \|_{A^{\sigma}} = \sum_{k = 0}^{\infty}(k + 1)^{\frac{1}{4}} \frac{\sigma^{k}}{k!} \| \partial_{x}^{k} f \|_{L^{\infty}}.
\]
Note that assumption (A2) is then equivalent to the statement $a \in A^{\sigma_{0}}$.  Obviously, these spaces are nested, so that
\[
A^{\sigma} \hookrightarrow A^{\sigma'}
\]
whenever $\sigma' \leq \sigma$.  Moreover, these spaces satisfy the following product estimates, which follow from Lemmas 2.2 and 2.3 of \cite{Wang}:
\begin{lemma}\label{AProd}
    Let $\sigma > 0$.  If $a \in A^{\sigma}$ and $u \in G^{\sigma}(\mathbb{R})$, then
    \[
    \| a u \|_{G^{\sigma}} \lesssim \| a \|_{A^{\sigma}} \| u \|_{G^{\sigma}}.
    \]
\end{lemma}

\begin{lemma}\label{AProd1}
    Let $\sigma > 0$, $0 < \delta \leq 1$, $b' \leq 0$, and $b \geq 0$.  Then for all $a \in A^{\sigma}$ and $u \in G^{\sigma}$, 
    \[
    \| a u \|_{X^{\delta}_{\sigma, b'}} \lesssim \| a \|_{A^{\sigma}} \| u \|_{X_{\sigma, b}^{\delta}}.
    \]
\end{lemma}
\noindent These spaces will be crucial in the construction of analytic solutions to the problem \eqref{kdvk}.

The final result we introduce in this section is the following basic estimate, which is Lemma 3.2 of \cite{dancona}:
\begin{lemma}\label{georgiev}
    If $a$ and $b$ are real numbers with $a \geq b$ and $a \geq 0$, then there exists a constant $C = C(a,b) > 0$ such that
    \[
    \langle x \pm y \rangle^{b} \leq C(a,b) \langle x \rangle^{a} \langle y \rangle^{b}.
    \]
    If, in addition, we have the inequality $|b| \leq a$, then $C$ can be chosen independently of $b$.
\end{lemma}

\subsection{Linear and Multilinear Estimates}\label{est}

Consider now the linear Cauchy problem
\begin{equation}\label{p2}
\begin{cases}
& \partial_{t}u + \alpha \partial_{x}^{5} u + \beta \partial_{x}^{3}u = N(x,t), \\
& u(x, 0)=u_{0}(x).
\end{cases}
\end{equation}
By using Duhamel's formula, we may write the solution
\begin{equation*}
u(x,t) = S(t)u_{0}(x) + \int_{0}^{t}S(t-t') N(x,t')\ dt',
\end{equation*}
where $S(t)$ is the solution operator, defined in terms of the Fourier transform by
\begin{equation*}
\widehat{S(t) f} = e^{-it(\alpha\xi^{5}-\beta\xi^{3})}\hat{f}.
\end{equation*}
By using the commutativity of Fourier multipliers, Lemma 2.11 and Proposition 2.12 of \cite{Tao}, it can be shown that solutions satisfy the following energy-type estimate:
\begin{lemma}\label{energy}
    If $u$ is a solution to equation \eqref{p2}, then for any smooth cutoff $\eta$, $\sigma > 0$ and $\frac{1}{2} < b < b' < 1$, we have the estimate
    \[
    \| \eta u \|_{X_{\sigma, b}} \lesssim_{\eta, b} \| u_{0} \|_{G^{\sigma}} + \delta^{b' - b}\| N \|_{X_{\sigma, b'-1}}.
    \]
\end{lemma}

Next, we state some necessary multilinear estimates which will be necessary to prove Theorem \ref{cor2}.
\begin{lemma}\label{bilinear}
    If $\sigma > 0$ and $b > \frac{1}{2}$ is sufficiently close to $\frac{1}{2}$, and $b' > \frac{1}{2}$ then
    \[
    \| \partial_{x} (u v) \|_{X_{\sigma, b - 1}} \lesssim \| u \|_{X_{\sigma, b'}} \| v \|_{X_{\sigma, b'}}.
    \]
\end{lemma}
\begin{proof}
The case $\sigma = 0$ is a special case of Corollary 3.3 of \cite{Jia}.  For the general case, we observe that
\[
\| \partial_{x} (u v) \|_{X_{\sigma, b-1}} = \| \xi e^{\sigma |\xi|} (1 + |\tau + \phi(\xi)|)^{b-1} \hat{u} * \hat{v} \|_{L^{2}_{\xi,\tau}}.
\]
By using the basic inequality
\[
e^{\sigma |\xi|} \leq e^{\sigma |\xi - \eta|} e^{\sigma |\eta|},
\]
it is easy to see that
\[
\| \xi e^{\sigma |\xi|} (1 + |\tau + \phi(\xi)|)^{b-1} \hat{u} * \hat{v} \|_{L^{2}_{\xi,\tau}} \leq \left\| |\xi| (1 + |\tau + \phi(\xi)|)^{b-1} u_{1} * v_{1} \right\|_{L^{2}_{\xi,\tau}},
\]
where
\[
\widehat{u_{1}} = e^{\sigma |\xi|} | \hat{u}(\xi) | \quad \text{and} \quad \widehat{v_{1}} = e^{\sigma |\xi|} | \hat{v}(\xi) |.
\]
The desired result then follows by applying the $\sigma = 0$ case to $u_{1}$ and $v_{1}$.
\end{proof}
By a similar analysis, we may prove the following trilinear estimate, which follows from Theorem 4.1 of \cite{Jia}:
\begin{lemma}\label{trilinear}
    Let $\sigma > 0$, $\frac{1}{2} < b < \frac{7}{10}$, and $b' > \frac{1}{2}$.  Then
    \[
    \| \partial_{x} \left( u_{1} u_{2} u_{3} \right) \|_{X_{\sigma, b-1}} \lesssim \| u_{1} \|_{X_{\sigma, b'}} \| u_{2} \|_{X_{\sigma, b'}} \| u_{3} \|_{X_{\sigma, b'}}.
    \]
\end{lemma}

To conclude our preliminary discussion, we observe that, due to the stability of $X_{\sigma, b}(\mathbb{R}^{2})$ spaces with respect to localization in time, all of the above estimates hold with $X_{\sigma, b}(\mathbb{R}^{2})$ replaced by the restricted space $X_{\sigma, b}^{\delta}(\mathbb{R}^{2})$.  With these preliminary tools ready, we may begin the proof of Theorem \ref{cor2}.


\section{Local Well-Posedness} \label{local}

For clarity, we restate Theorem \ref{cor2} in the following, more explicit form:
\begin{proposition}[Local Well-posedness in $G^{\sigma_{0}}(\mathbb{R})$]\label{cor2restate}
	Let $\sigma_{0} > 0$.  Then for each $u_{0}^{*} \in G^{\sigma_{0}}(\mathbb{R})$, there exists a number $b > \frac{1}{2}$, a time $\delta > 0$, and an open ball $B \subset G^{\sigma_{0}}(\mathbb{R})$ containing $u_{0}^{*}$ such that for each $u_{0} \in B$, there exists a unique strong $G^{\sigma_{0}}(\mathbb{R})$ solution $u \in X_{\sigma_{0},b}(\mathbb{R}^{2})$ to the Cauchy problem \eqref{kdvk}.  Moreover, the solution satisfies the bound
	\[
	\| u \|_{L^{\infty}_{t} G^{\sigma_{0}}_{x}} \leq 2 C \| u_{0} \|_{G^{\sigma_{0}}}
	\]
	for some $C > 0$.  Finally, the solution map $u_{0} \mapsto u$ is Lipschitz continuous.
\end{proposition}

The proof is a standard fixed point argument, where we define an operator $\Phi: X^{\delta}_{\sigma_{0}, b}(\mathbb{R}^{2}) \rightarrow X^{\delta}_{\sigma_{0}, b}(\mathbb{R}^{2})$, which we show to be a contraction.  We begin with the question of existence and uniqueness.

\subsection{Existence and Uniqueness}

Let $u_{0}^{*}\in G^{\sigma_{0}}(\mathbb{R})$, and fix $\delta > 0$ to be determined later.  Define
\[
B = \left\{ u_{0} \in G^{\sigma} : \| u_{0} \|_{G^{\sigma}} < 2 \| u_{0}^{*} \|_{G^{\sigma}} \right\}.
\]
Next, choose $u_{0} \in B$.  For $u \in X_{\sigma_{0}, b}^{\delta}(\mathbb{R}^{2})$, define
\[
\Phi(u) = \psi_{\delta}(t) S(t) u_{0} + \psi_{\delta}(t) \int_{0}^{t} S(t - t') N(u(t'))\ dt',
\]
where
\[
N(u) = \mu \partial_{x}(u^{2}) + \lambda \partial_{x}(u^{3}) + a(x) u.
\]
We first show that $\Phi : X^{\delta}_{\sigma_{0},b}(\mathbb{R}^{2}) \rightarrow X^{\delta}_{\sigma_{0},b}(\mathbb{R}^{2})$.  Applying Lemmas \ref{AProd}, \ref{energy}, \ref{bilinear} and \ref{trilinear}, we easily see that
\[
\| \Phi(u) \|_{X^{\delta}_{\sigma_{0},b}} \lesssim \| u_{0} \|_{G^{\sigma_{0}}} + \delta^{b' - b} \| N(u) \|_{X^{\delta}_{\sigma_{0}, b'-1}}.
\]
Recalling the definition of $N(u)$, we may apply Lemmas \ref{AProd}, \ref{bilinear} and \ref{trilinear}, to see that
\[
\| N(u) \|_{X^{\delta}_{\sigma_{0}, b'-1}} \lesssim \| a \|_{A^{\sigma}} \| u \|_{X_{\sigma_{0}, b}} + \| u \|_{X_{\sigma_{0}, b}}^{2} + \| u \|_{X_{\sigma_{0}, b}}^{3}
\]
for $b$, $b'$ sufficiently small.  We conclude that $\Phi(u) \in X_{\sigma, b}(\mathbb{R}^{2})$.

Next, we show that the map $\Phi$ is a contraction on the ball $B(0,R) \subset X^{\delta}_{\sigma_{0}, b}(\mathbb{R}^{2})$, where $R = 2 C \| u_{0} \|_{G^{\sigma_{0}}}$.  For $u, v \in B(0,R)$, we again apply Lemmas \ref{AProd}, \ref{energy}, \ref{bilinear} and \ref{trilinear} to see that
\begin{align*}
    \| \Phi(u) - \Phi(v) \|_{X^{\delta}_{\sigma_{0},b}} \leq C M \delta^{b'- b}  \| u - v \|_{X^{\delta}_{\sigma_{0}, b}},
\end{align*}
where
\begin{equation}\label{emm}
M = \| a \|_{A^{\sigma_{0}}} + \| u \|_{X^{\delta}_{\sigma_{0}, b}} + \| v \|_{X^{\delta}_{\sigma_{0}, b}} + \| u \|_{X^{\delta}_{\sigma_{0}, b}}^{2} + \| v \|_{X^{\delta}_{\sigma_{0}, b}}^{2}.
\end{equation}
Thus, if we choose
\[
\delta < C^{-\frac{1}{b' - b}} \left( \| a \|_{A^{\sigma_{0}}} + \| u \|_{X^{\delta}_{\sigma_{0}, b}} + \| v \|_{X^{\delta}_{\sigma_{0}, b}} + \| u \|_{X^{\delta}_{\sigma_{0}, b}}^{2} + \| u \|_{X^{\delta}_{\sigma_{0}, b}}^{2} \right)^{- \frac{1}{b' - b}},
\]
then $\Phi$ will satisfy
\[
\| \Phi(u) - \Phi(v) \|_{X^{\delta}_{\sigma_{0}, b}} < \| u - v \|_{X^{\delta}_{\sigma_{0}, b}}.
\]
It follows that $\Phi$ is a contraction.

To complete the proof, we note that the Contraction Mapping Theorem then implies that there exists a unique fixed point $u$, which will be the solution to the integral equation $\Phi(u) = u$.  By construction the solution will satisfy the estimate
\begin{equation}\label{embedding}
\| u \|_{X^{\delta}_{\sigma_{0, b}}} \leq 2 C \| u_{0} \|_{G^{\sigma_{0}}}.
\end{equation}
The desired estimate in Proposition \ref{cor2restate} then follows from the embedding $$X^{\delta}_{\sigma_{0}, b}(\mathbb{R}^{2}) \hookrightarrow L^{\infty}_{t} G^{\sigma_{0}}_{x}.$$  By Theorem \ref{wiener}, it follows that the desired holomorphic extension exists on the strip $S_{\sigma_{0}}$ for $t \in [0,\delta]$.  We then conclude that the analyticity of the solutions persists, at least up to time $t = \delta$.

\subsection{Continuous Dependence of the Initial Data}

Next, we would like to prove that the solution map $u_{0} \mapsto u$ is Lipschitz continuous.  By the embedding \eqref{embedding}, this will follow from the following result:
\begin{lemma} \label{cha1lem12}
	Let $\sigma$ and $b$ be as in Proposition \ref{cor2restate}, and let $u, v$, respectively, be solutions to the problem \eqref{kdvk} corresponding to initial data $u_{0}, v_{0}$, respectively, from Proposition \ref{cor2restate}.  Then there exists a constant $C > 0$ such that
	\begin{equation}
	\| u - v \|_{X^{\delta}_{\sigma, b}} \leq C \| u_{0}-v_{0}\|_{G^{\sigma}}.
	\end{equation}
\end{lemma}
\begin{proof}
	If $u$ and $v$ are two solutions to \eqref{kdvk} corresponding to initial data $u_{0}$ and $v_{0}$, we may apply the energy estimate in Lemma \ref{energy} to see that
	\begin{align*}
	\|u-v\|_{X_{\sigma, b}^{\delta}} & = \|\Phi(u)-\Phi(v)\|_{X_{\sigma, b}^{\delta}} \\
	& \leq C \| u_{0} - v_{0} \|_{G^{\sigma}} + C \delta^{b' - b} \| N(u) - N(v) \|_{X^{\delta}_{\sigma, b' - 1}}.
	\end{align*}
	As was done previously, we apply Lemmas \ref{AProd}, \ref{energy}, \ref{bilinear} and \ref{trilinear} to see that
	\[
	\| N(u) - N(v) \|_{X^{\delta}_{\sigma, b' - 1}} \leq C M \| u - v \|_{X^{\delta}_{\sigma, b}},
	\]
	where $M$ is as in equation \eqref{emm}.  Thus
	\[
	\|u-v\|_{X_{\sigma, b}^{\delta}} \leq C \| u_{0} - v_{0} \|_{G^{\sigma}} + C M \delta^{b' - b} \| u - v \|_{X^{\delta}_{\sigma, b}}.
	\]
	
	Using the estimate in equation \eqref{embedding}, we may conclude that if $\delta$ is sufficiently small, then
	\[
	\|u-v\|_{X_{\sigma, b}^{\delta}} \leq C \| u_{0} - v_{0} \|_{G^{\sigma}} + \frac{1}{2} \| u - v \|_{X^{\delta}_{\sigma, b}},
	\]
	or
	\[
	\| u - v \|_{X^{\delta}_{\sigma, b}} \leq 2 C \| u_{0} - v_{0} \|_{G^{\sigma}}.
	\]
	The desired result follows.
\end{proof}


\section{A Lower Bound for the Radius of Spatial Analyticity}\label{sect1.5}

In this section, we prove Theorem \ref{global}.  By the well-known embedding $G^{\sigma_{0}}(\mathbb{R}) \hookrightarrow G^{\sigma}(\mathbb{R})$, which holds for $\sigma \leq \sigma_{0}$, our local solution from Theorem \ref{cor2} also belongs to $G^{\sigma}(\mathbb{R})$ for any $\sigma \leq \sigma_{0}$.  By a standard argument, we may extend our solutions to an arbitrarily long time interval $[0,T]$ if we can show that the $G^{\sigma}(\mathbb{R})$ norm of $u$ remains uniformly bounded in $[0,T]$.  Thus, to prove Theorem \ref{global}, it suffices to obtain the estimate \eqref{exp}.  As is often the case, this will be accomplished by deriving an approximate conservation law, which will be used to control the growth of the $G^{\sigma_{1}}(\mathbb{R})$ norm.

\subsection{Approximate Conservation Law}
We start by recalling that, for smooth, compactly supported solutions to \eqref{kdvk}, we have the identity
\[
\dfrac{d}{dt} \left[\int_{\mathbb{R}} u^{2}(x,t) dx \right] + 2 \int_{\mathbb{R}} d(x) u^{2}(x, t) dx = 0.
\]
Under the assumption (A1), we may apply Gr\"onwall's inequality to deduce that
\begin{equation}\label{cons1}
\| u(t)\|_{L^{2}} \leq \| u_{0} \|_{L^{2}} e^{- \gamma t}.
\end{equation}
By a limiting argument, the result also holds for arbitary solutions in $L^{2}(\mathbb{R})$.  From this, we obtain global well-posedness in $L^{2}(\mathbb{R})$.

In the same vein, we would like to prove an analogous result for the $G^{\sigma}(\mathbb{R})$ norms of solutions.  Explicitly, we aim to prove the following result:
\begin{proposition}\label{cha1the3}
	Let $0 < \sigma \leq \sigma_{0}$, $\kappa \in [0, 1]$, $\varrho \in [0, \frac{1}{4}]$ and let $0 < \delta < 1$ be as in Theorem \ref{cor2}.  If $u \in X^{\delta}_{\sigma, b}(\mathbb{R}^{2})$ is a solution to the Cauchy problem (\ref{kdvk}) on the time interval $[0, \delta]$, then the following estimate holds:
	\begin{align*}
	\| u(t)\|^{2}_{G^{\sigma}} & \leq e^{-2\gamma t} \| u_{0}\|^{2}_{G^{\sigma}} + C_{1} \sigma^{\kappa} \| u_{0} \|^{3}_{G^{\sigma}} + C_{2} \sigma^{\varrho} \| u_{0} \|^{4}_{G^{\sigma}}\\ & \quad + C_{3} \sigma \| a \|_{A^{\sigma}}\| u_{0} \|^{2}_{G^{\sigma}} + C_{4} \| a \|_{L^{\infty}} \|u_{0}\|_{L^{2}} \|u_{0} \|_{G^{\sigma}}.
	\end{align*}
\end{proposition}
\begin{proof}
Let $\kappa, \sigma, T, b$ and $u$ be as in the statement of the theorem.  For simplicity of notation, let us define a pseudodifferential operator $\Lambda^{\sigma}$ in terms of its Fourier transform:
\[
\widehat{\Lambda^{\sigma}u}(\xi, t) = e^{\sigma|\xi |}\widehat{u}(\xi, t).
\]
Additionally, it will be useful to define
\begin{align*}
\Delta(u) & = \partial_{x} \left[ (\Lambda^{\sigma}u)^{2} - \Lambda^{\sigma} (u^{2}) \right], \\
\Theta(u) & = \partial_{x} \left[ (\Lambda^{\sigma} u)^{3} - \Lambda^{\sigma} (u^{3}) \right], \\
\Gamma(u) & = \left[ d \Lambda^{\sigma} u - \Lambda^{\sigma} (du) \right].
\end{align*}
When there is no danger of confusion, we will omit the $u$ from the notation.

To start, define the auxiliary function $V(t, x) = \Lambda^{\sigma}u(t, x)$. Since $u$ is real-valued, then $V$ must also be real-valued.  Applying the pseudodifferential operator $\Lambda^{\sigma}$ to equation (\ref{p1}), it is easily seen that we obtain
\begin{equation*}
\partial_{t} V + \alpha \partial_{x}^{5} V + \beta \partial_{x}^{3} V + \mu \Lambda^{\sigma} \partial_{x} u^{2} + \lambda \Lambda^{\sigma}\partial_{x}u^{3} + \Lambda^{\sigma} (d u) = 0.
\end{equation*}
Using the notation above, it is easy to see that this equation is equivalent to
\begin{equation*}
\partial_{t} V + \alpha \partial_{x}^{5} V + \beta \partial_{x}^{3} V + \mu \partial_{x} V^{2} + \lambda \partial_{x} V^{3} + d V = \mu \Delta + \lambda \Theta + \Gamma.
\end{equation*}
Multiplying this equation by $V$ and integrating in space, we obtain
\begin{equation}\label{1.1.23}
\begin{aligned}
& \int_{\mathbb{R}}V\partial_{t} V\ dx + \alpha \int_{\mathbb{R}}V \partial_{x}^{5} V\ dx + \beta \int_{\mathbb{R}}V \partial_{x}^{3} V\ dx + \mu \int_{\mathbb{R}} V \partial_{x}V^{2}\ dx\ + \\
& \ + \lambda \int_{\mathbb{R}}V\partial_{x} V^{3}\ dx + \int_{\mathbb{R}} d V^{2} \ dx = \mu \int_{\mathbb{R}} V \Delta\ dx+\lambda \int_{\mathbb{R}}\ V \Theta \  dx + \int_{\mathbb{R}} V \Gamma \ dx.
\end{aligned}
\end{equation}

To simplify this equation, let us recall the following fact\footnote{See the footnote on page 1018 of \cite{Selberg1}.}: for any $j \in \mathbb{N}$, we have
\[
|\partial_{x}^{j} V(x,t)| \rightarrow 0 \quad \text{as} \quad |x| \rightarrow \infty.
\]
We may thus apply integration by parts to see that
\[
\int_{\mathbb{R}}V \partial_{x}^{5} V\ dx = \int_{\mathbb{R}} V \partial_{x}^{3} V\ dx = \int_{\mathbb{R}} V \partial_{x}V^{2}\ dx\ = \int_{\mathbb{R}}V\partial_{x} V^{3}\ dx = 0.
\]
Equation \eqref{1.1.23} can be thus be rewritten as
\[
\frac{1}{2}\frac{d}{dt}\int_{\mathbb{R}}V^{2} dx+\int_{\mathbb{R}} d(x) V^{2} \ dx = \mu \int_{\mathbb{R}} V \Delta \ dx + \lambda \int_{\mathbb{R}} V \Theta\  dx + \int_{\mathbb{R}} V \Gamma \ dx.
\]
Since $d(x) \geq \gamma > 0$ for all $x \in \mathbb{R}$ we get
\begin{equation}\label{gronsetup}
\frac{d}{dt} \int_{\mathbb{R}}V^{2} \  dx + 2 \gamma \int_{\mathbb{R}} V^{2} \ dx \leq 2 \mu \int_{\mathbb{R}} V \Delta \ dx + 2 \lambda \int_{\mathbb{R}} V \Theta \  dx + 2 \int_{\mathbb{R}} V \Gamma \ dx.
\end{equation}
Define
\[
M(t) = 2 \mu \int_{\mathbb{R}} V \Delta \ dx + 2 \lambda \int_{\mathbb{R}} V \Theta \  dx + 2 \int_{\mathbb{R}} V \Gamma \ dx.
\]
Equation \eqref{gronsetup} can then be written as
\[
\frac{d}{dt} \| u(t) \|_{G^{\sigma}}^{2} \leq - 2 \gamma \| u(t) \|_{G^{\sigma}}^{2} + M(t).
\]
Gr\"onwall's inequality then yields
\begin{equation}\label{gron}
\| u(t)\|^{2}_{G^{\sigma}} \leq e^{-2\gamma t} \| u(0) \|^{2}_{G^{\sigma}} + \int_{0}^{t} M(s) e^{-\gamma (t-s)}\ ds
\end{equation}
for $t \in [0,\delta]$.

To proceed, let us observe that for $t \in [0,\delta]$, we may rewrite the integral on the right-hand side of equation \eqref{gron} as
\begin{align*}
\int_{0}^{t} M(s) e^{-\gamma (t - s)}\ ds & = \int_{\mathbb{R}^{2}} \chi_{[0,\delta]}(t) \left[ 2 \mu V \Delta + 2 \lambda  V \Theta + 2 V \Gamma \right] e^{- \gamma (t - s)} \ dx ds \\
& = \mathcal{I}_{1} + \mathcal{I}_{2} + \mathcal{I}_{3}.
\end{align*}
We first estimate $\mathcal{I}_{1}$.  Recalling that $V$ is real-valued, it follows from Parseval's identity that
\begin{align*}
    |\mathcal{I}_{1}| & = 2 |\mu| \left| \int_{\mathbb{R}^{2}} \chi_{[0,\delta]}(t) e^{-\gamma(t - s)} V \chi_{[0,\delta]} \Delta\ dx ds \right| \\
    & = 2 |\mu| \left| \int_{\mathbb{R}^{2}} \overline{\chi_{[0,\delta]}(t) e^{-\gamma(t - s)} V} \chi_{[0,\delta]} \Delta\ dx ds \right| \\
    & = 2 |\mu| \left| \int_{\mathbb{R}^{2}} \overline{\widehat{\chi_{[0,\delta]} e^{-\gamma(t - s)} V}} \widehat{\chi_{[0,\delta]} \Delta}\ d\xi d\tau \right| \\
    & = 2 |\mu| \left| \int_{\mathbb{R}^{2}} (1 + |\tau + \phi(\xi)|)^{1-b} \overline{\widehat{\chi_{[0,\delta]} e^{-\gamma(t - s)} V}} (1 + |\tau + \phi(\xi)|)^{b-1}\widehat{\chi_{[0,\delta]} \Delta}\ d\xi d\tau \right|
\end{align*}
H\"older's inequality then yields that
\[
| \mathcal{I}_{1} | \leq 2 |\mu| \| \chi_{[0,\delta]} e^{-\gamma(t - \cdot)} V \|_{X_{0,1 - b}} \| \chi_{[0,\delta]} \Delta \|_{X_{0,b-1}}.
\]
An analogous computation will show that
\[
| \mathcal{I}_{2} | \leq 2 |\lambda| \| \chi_{[0,\delta]} e^{-\gamma(t - \cdot)} V \|_{X_{0,1 - b}} \| \chi_{[0,\delta]} \Theta \|_{X_{0,b-1}}.
\]

Next, observe that we have both $-\frac{1}{2} < b - 1 < \frac{1}{2}$ and $-\frac{1}{2} < 1 - b < \frac{1}{2}$. Therefore, one can use apply Proposition 2.1 and Lemma 3.1 of \cite{Wang} to obtain	
\begin{align*}
    | \mathcal{I}_{1} | & \leq 2 C_{1} |\mu| \| V \|_{X^{\delta}_{0,1 - b}} \| \Delta \|_{X^{\delta}_{0,b-1}} \\
    & \leq 2 C_{1} |\mu| \| u \|_{X^{\delta}_{\sigma, b}} \| \Delta \|_{X^{\delta}_{0,b - 1}},
\end{align*}
and
\begin{equation}\label{thetanorm}
\begin{aligned}
    | \mathcal{I}_{2} | & \leq 2 C_{2} | \lambda | \| V \|_{X^{\delta}_{0,1 - b}} \| \Theta \|_{X^{\delta}_{0,b - 1}} \\
    & \leq 2 C_{2} |\mu| \| u \|_{X^{\delta}_{\sigma, b}} \| \Theta \|_{X^{\delta}_{0,b - 1}}.
\end{aligned}
\end{equation}
To estimate the norms of $\Delta$ and $\Theta$, we apply the following well-known trick: for any $\kappa \in [0,1]$, the inequality
\begin{equation}\label{expest1}
e^{\sigma |\xi - \xi'|}e^{\sigma |\xi'|} - e^{\sigma |\xi|} \leq C \min\{ |\xi - \xi'|, |\xi'| \}^{\kappa} \sigma^{\kappa} e^{\sigma |\xi - \xi'|}e^{\sigma |\xi'|}
\end{equation}
holds for some $C > 0$.  Moreover, it is easy to see that
\begin{equation}\label{expest2}
\min\{ |\xi - \xi'|, |\xi'| \} \lesssim \frac{\langle \xi - \xi' \rangle \langle \xi' \rangle}{\langle \xi \rangle}.
\end{equation}
Applying Corollary 3.3 of \cite{Jia}, it then follows that
\begin{align*}
    \| \Delta \|_{X^{\delta}_{0, b - 1}} & \lesssim \sigma^{\kappa} \left\| \frac{i \xi}{\langle \xi \rangle^{\kappa} \langle \tau + \phi(\xi) \rangle^{1 - b}} \widehat{\langle \nabla \rangle^{\kappa} V} * \widehat{\langle \nabla \rangle^{\kappa} V} \right\|_{L^{2}_{\xi, \tau}} \\
    & \lesssim \sigma^{\kappa} \| V \|_{X^{\delta}_{0, b}}^{2} \\
    & \lesssim \sigma^{\kappa} \| u \|_{X^{\delta}_{\sigma, b}}.
\end{align*}
where
\[
\langle \nabla \rangle^{\kappa} V = \mathfrak{F}^{-1} \left( \langle \xi \rangle^{\kappa} \tilde{V}(\xi,\tau) \right).
\]
We thus conclude that
\[
| \mathcal{I}_{1} | \leq C_{1} \sigma^{\kappa} \| u \|_{X^{\delta}_{\sigma, b}}^{3}
\]
for some $C_{1} > 0$.

Next, we estimate $\mathcal{I}_{2}$.  From equation \eqref{thetanorm}, it suffices to estimate $\| \Theta \|_{X^{\delta}_{0, b - 1}}$.  To estimate the norm of $\Theta$, observe that
\begin{align*}
\widehat{\Theta} & = i \xi \left[ \hat{V}*\hat{V}*\hat{V} - e^{\sigma |
\xi|} (\hat{u} * \hat{u} * \hat{u}) \right] \\
& = i \xi \int_{\mathbb{R}^{2}} \left( e^{\sigma |\xi - \xi_{1}|} e^{\sigma |\xi_{1} - \xi_{2}|} e^{\sigma |\xi_{2}|} - e^{\sigma |\xi|} \right) \hat{u}(\xi - \xi_{1}) \hat{u}(\xi_{1} - \xi_{2}) \hat{u}(\xi_{2})\ d\xi_{1} \xi_{2}.
\end{align*}
Applying the estimates in equations \eqref{expest1} and \eqref{expest2} twice, we obtain
\begin{align*}
e^{\sigma |\xi - \xi_{1}|} e^{\sigma |\xi_{1} - \xi_{2}|} e^{\sigma |\xi_{2}|} - e^{\sigma |\xi|} & \lesssim \sigma^{\varrho} \left( \min\{ |\xi - \xi_{1}|, |\xi_{1}| \}^{\varrho} + \min\{ |\xi_{1} - \xi_{2}|, |\xi_{2}| \}^{\varrho} \right) \times \\
& \quad \times e^{\sigma |\xi - \xi_{1}|} e^{\sigma |\xi_{1} - \xi_{2}|} e^{\sigma |\xi_{2}|} \\
& \lesssim \sigma^{\varrho}\left(\frac{\langle \xi - \xi_{1} \rangle^{\varrho} \langle \xi_{1} \rangle^{\varrho}}{\langle \xi \rangle^{\varrho}} + \frac{\langle \xi_{1} - \xi_{2} \rangle^{\varrho} \langle \xi_{2} \rangle^{\varrho}}{\langle \xi_{1} \rangle^{\varrho}} \right) \times \\ 
& \quad \times e^{\sigma |\xi - \xi_{1}|} e^{\sigma |\xi_{1} - \xi_{2}|} e^{\sigma |\xi_{2}|}.
\end{align*}
If we now apply Lemma \ref{georgiev}, then we may obtain the estimates
\begin{align*}
    \frac{\langle \xi - \xi_{1} \rangle^{\varrho} \langle \xi_{1} \rangle^{\varrho}}{\langle \xi \rangle^{\varrho}} & \lesssim \frac{\langle \xi - \xi_{1} \rangle^{\varrho} \langle \xi_{1} - \xi_{2} \rangle^{\varrho} \langle \xi_{2} \rangle^{\varrho}}{\langle \xi \rangle^{\varrho}}, \\
    \frac{\langle \xi_{1} - \xi_{2} \rangle^{\varrho} \langle \xi_{2} \rangle^{\varrho}}{\langle \xi_{1} \rangle^{\varrho}} & \lesssim \frac{\langle \xi - \xi_{1} \rangle^{\varrho} \langle \xi_{1} - \xi_{2} \rangle^{\varrho} \langle \xi_{2} \rangle^{\varrho}}{\langle \xi \rangle^{\varrho}}.
\end{align*}
We thus conclude that
\[
e^{\sigma |\xi - \xi_{1}|} e^{\sigma |\xi_{1} - \xi_{2}|} e^{\sigma |\xi_{2}|} - e^{\sigma |\xi|} \lesssim \sigma^{\varrho} \frac{\langle \xi - \xi_{1} \rangle^{\varrho} \langle \xi_{1} - \xi_{2} \rangle^{\varrho} \langle \xi_{2} \rangle^{\varrho}}{\langle \xi \rangle^{\varrho}}e^{\sigma |\xi - \xi_{1}|} e^{\sigma |\xi_{1} - \xi_{2}|} e^{\sigma |\xi_{2}|}.
\]

Proceeding as we did in the estimate for $\mathcal{I}_{1}$, we may thus use the estimates above to see that
\[
\| \Theta \|_{X^{\delta}_{0, b-1}} \lesssim \sigma^{\varrho} \| \langle \nabla \rangle^{-\varrho} \partial_{x} \left( W^{3} \right) \|_{X^{\delta}_{0, b-1}},
\]
where
\[
W = \langle \nabla \rangle^{\varrho} \Lambda^{\sigma}u.
\]
Applying Theorem 4.1 of \cite{Jia}, we see that
\begin{align*}
    \| \langle \nabla \rangle^{-\varrho} \partial_{x} \left( W^{3} \right) \|_{X^{\delta}_{0, b-1}} & \lesssim \| \langle \nabla \rangle^{-\varrho} W \|_{X^{\delta}_{0, b}}^{3} \\
    & \lesssim \| u \|_{X^{\delta}_{\sigma, b}}^{3}.
\end{align*}
We thus conclude that
\[
| \mathcal{I}_{2} | \leq C_{2} \sigma^{\varrho} \| u \|_{X^{\delta}_{\sigma, b}}^{4}
\]
for some $C_{2} > 0$.

Finally, it remains to estimate $\mathcal{I}_{3}$.  From the definition, we have
\[
\mathcal{I}_{3} \leq 2 \int_{0}^{t} e^{-2 \gamma (t - s)}\|\Gamma(s)\|_{L^{2}(\mathbb{R})}\| u(s) \|_{G^{\sigma}(\mathbb{R})} ds \\
\]
If we apply Lemma 3.3 of \cite{Wang}, we obtain that
\[
\| \Gamma(s) \|_{L^{2}} \lesssim \frac{\sigma}{\sigma_{0}} \| a \|_{A^{\sigma_{0}}} \| u(s) \|_{G^{\sigma}} + \| a \|_{L^{\infty}} \| u(s) \|_{L^{2}}
\]
for $0 \leq \sigma \leq \sigma_{0}$.  This yields the estimate
\[
\mathcal{I}_{3}  \lesssim \int_{0}^{t} e^{- 2 \gamma (t - s)} \left( \frac{\sigma}{\sigma_{0}} \| a \|_{A^{\sigma_{0}}} \| u(s) \|_{G^{\sigma}}^{2} + \| a \|_{L^{\infty}} \| u(s) \|_{L^{2}} \| u(s) \|_{G^{\sigma}} \right)\ ds,
\]
which we rewrite as
\[
\mathcal{I}_{3} \lesssim  \int_{\mathbb{R}^{2}} \chi_{[0,\delta]}(t) e^{-2 \gamma (t - s)} \left( \frac{\sigma}{\sigma_{0}} \| a \|_{A^{\sigma_{0}}} \| u(s) \|_{G^{\sigma}}^{2} + \| a \|_{L^{\infty}} \| u(s) \|_{L^{2}} \| u(s) \|_{G^{\sigma}} \right)\ ds.
\]
Proceeding as we did for $\mathcal{I}_{1}$ and $\mathcal{I}_{2}$, we may again apply Proposition 2.1 and Lemma 3.1 of \cite{Wang}, as well as equations \eqref{embedding} and \eqref{cons1}, we obtain
\[
\mathcal{I}_{3} \lesssim \sigma \| a \|_{A^{\sigma_0}} \| u_{0} \|^{2}_{G^{\sigma}} + \| a \|_{L^{\infty}} \| u_{0} \|_{L^{2}} \| u_{0} \|_{G^{\sigma}}.
\]
We may therefore conclude  that
\begin{align*}
\| u(t)\|^{2}_{G^{\sigma}}& 
\leq e^{-2\gamma t}\| u_{0} \|^{2}_{G^{\sigma}} + \mathcal{I}_{1} + \mathcal{I}_{2} + \mathcal{I}_{3} \\
& \leq e^{-2\gamma t} \| u_{0}\|^{2}_{G^{\sigma}} + C_{1} \sigma^{\kappa} \| u_{0} \|^{3}_{G^{\sigma}} + C_{2} \sigma^{\varrho} \| u_{0} \|^{4}_{G^{\sigma}}\\ & \quad + C_{3} \sigma \| a \|_{A^{\sigma_0}}\| u_{0} \|^{2}_{G^{\sigma}} + C_{4} \| a \|_{L^{\infty}} \|u_{0}\|_{L^{2}} \|u_{0} \|_{G^{\sigma}}
\end{align*}
The proof is now complete.
\end{proof}

\subsection{Proof of Theorem \ref{global}}

Using the results from the previous section, we may now begin the proof of Theorem \ref{global}.  The main idea of the proof is to use the almost conservation law of Proposition \ref{cha1the3} to control the growth of the $G^{\sigma}(\mathbb{R})$ norm of $u$ on intervals of the form $[0, n \delta]$.  Since we may find $n$ such that
\[
[0, n \delta] \subset [0, T] \subset [0, (n+1)\delta],
\]
it suffices to show that the norm remains bounded on $[0, n \delta]$ for any $n$.  In fact, it suffices to show that the norm of the solution is bounded at $t = n \delta$ for all $n \in \mathbb{N}$.  This is the content of the following theorem.
\begin{lemma}
    Let $u$ be the local solution from Theorem \ref{cor2}.  Then there exists $\sigma_{1} > 0$, such that for each $n \in \mathbb{N}$, there exist constants $D_{n} > 0$, depending on $\delta$, $\| u_{0} \|_{G^{\sigma}}$, and $\| a \|_{L^{\infty}}$, satisfying
    \[
    \| u(n \delta) \|_{G^{\sigma_{1}}}^{2} \leq \| u_{0} \| _{G^{\sigma_{0}}}^{2} + D_{n} \| u_{0} \|_{L^{2}}^{2}.
    \]
\end{lemma}
\begin{proof}
By induction on $n$.  In the $n = 1$ case, we first apply Proposition \ref{cha1the3} to see that, for $t > 0$, we have
\begin{equation}\label{delest}\begin{aligned}
	\| u(\delta)\|^{2}_{G^{\sigma}} & \leq e^{-2\gamma \delta} \| u_{0}\|^{2}_{G^{\sigma}} + C_{1} \sigma^{\kappa} \| u_{0} \|^{3}_{G^{\sigma}} + C_{2} \sigma^{\varrho} \| u_{0} \|^{4}_{G^{\sigma}}\\ & \quad + C_{3} \sigma \| a \|_{A^{\sigma}}\| u_{0} \|^{2}_{G^{\sigma}} + C_{4} \| a \|_{L^{\infty}} \|u_{0}\|_{L^{2}} \|u_{0} \|_{G^{\sigma}}.
\end{aligned}
\end{equation}
Focusing on the final term on the right-hand side of this inequality, observe that for $t > 0$, we have
\[
C_{4} \| a \|_{L^{\infty}} \|u_{0}\|_{L^{2}} \|u_{0} \|_{G^{\sigma}} \leq \frac{1 - e^{-2 \gamma \delta}}{2} \| u_{0} \|_{G^{\sigma}}^{2} + \frac{C_{4}^{2} \| a \|_{L^{\infty}}^{2}}{2(1 - e^{-2 \gamma \delta})} \| u_{0} \|_{L^{2}}^{2}.
\]
The base case is then satisfied if $\sigma$ is chosen so that
\begin{equation}\label{smallness}\begin{aligned}
    C_{1} \sigma^{\kappa} \| u_{0} \|_{G^{\sigma}} & \leq \frac{1 - e^{- 2 \gamma \delta}}{6}, \\
    C_{2} \sigma^{\rho} \| u_{0} \|_{G^{\sigma}}^{2} & \leq \frac{1 - e^{- 2 \gamma \delta}}{6}, \\
    C_{3} \sigma \| a \|_{A^{\sigma}} & \leq \frac{1 - e^{- 2 \gamma \delta}}{6},
\end{aligned}\end{equation}
and
\[
D_{1} = \frac{C_{4}^{2} \| a \|_{L^{\infty}}^{2}}{2 (1 - e^{-2 \gamma \delta})}.
\]

Now, assume the result holds for $n = k$.  Setting $u(k \delta)$ as initial data $\tilde{u}_{0}$ and applying the estimate from the base case, we have
\begin{align*}
    \| u((k+1)\delta) \|_{G^{\sigma}}^{2} & \leq \| \tilde{u}_{0} \|_{G^{\sigma}}^{2} + D_{1} \| \tilde{u}_{0} \|_{L^{2}} \\
    & \leq \| u(k\delta) \|_{G^{\sigma}}^{2} + D_{1} \| u(k \delta) \|_{L^{2}}^{2}.
\end{align*}
Using equation \eqref{cons1} and the inductive hypothesis, this becomes
\begin{equation}\label{almostdone}
\| u((k+1)\delta) \|_{G^{\sigma}}^{2} \leq \| u_{0} \|_{G^{\sigma}}^{2} + D_{k} \| u_{0} \|_{L^{2}}^{2} + e^{- 2 k \delta \gamma} \| u_{0} \|_{L^{2}}^{2}.
\end{equation}
Setting
\[
D_{k+1} = D_{k} + e^{- 2 k \delta \gamma}
\]
yields the desired result.
\end{proof}

To complete the proof of Theorem \ref{global}, we observe that the estimate in equation \eqref{almostdone} implies that
\[
\| u(k\delta) \|_{G^{\sigma}} \leq (2 + D_{1})^{\frac{1}{2}} \left( \sum_{j = 0}^{k-1} e^{-2 j \delta \gamma} \right)^{\frac{1}{2}} \| u_{0} \|_{G^{\sigma}}.
\]
Letting $k \rightarrow \infty$ in the sum above yields a convergent series, which implies that
\[
\| u(k\delta) \|_{G^{\sigma}} \leq C \| u_{0} \|_{G^{\sigma}}
\]
for some $C > 0$.  This implies that we may apply the local well-posedness result repeatedly at $t = n \delta$, and the resulting solution will satisfy
\begin{equation}\label{globalfinal}
\| u(t) \|_{G^{\sigma}} \leq C \| u_{0} \|_{G^{\sigma}}
\end{equation}
for any $\sigma$ satisfying the inequalities in equation \eqref{smallness} and $\sigma \leq \sigma_{0}$.
Let $\sigma_{1}$ be any such number.

Next, observe that
\begin{align*}
    \| u(t) \|_{G^{\frac{\sigma_{1}}{2}}} & \leq \| u(t) \|_{L^{2}}^{\frac{1}{2}} \| u(t) \|_{G^{\sigma_{1}}}^{\frac{1}{2}} \\
    & \leq C e^{- \frac{t \gamma}{2}} \| u_{0} \|_{G^{\sigma_{0}}}
\end{align*}
by equations \eqref{cons1} and \eqref{globalfinal}.  Setting
\[
\tilde{\sigma_{0}} = \frac{\sigma_{1}}{2}
\]
yields equation \eqref{exp}.  This completes the proof of Theorem \ref{global}.

\end{document}